\documentclass[review]{elsarticle}
\usepackage{geometry}
\usepackage{lineno,hyperref}
\usepackage{dsfont}
\usepackage{tikz}
\usepackage{amsmath,amssymb,amsfonts,amsthm}

\usepackage{algorithm}
\usepackage{bm}
\usepackage{bookmark}
\usepackage{bm}
\modulolinenumbers[5]
\usepackage[all]{xy}
\usepackage{graphics}
\usepackage{multirow}
\graphicspath{{Figures/}}

\newtheorem{thm}{Theorem}[section]
\newtheorem{lem}[thm]{Lemma}

\theoremstyle{definition}
\newtheorem{dfn}[thm]{Definition}
\newdefinition{rmk}[thm]{Remark}

\theoremstyle{remark}
\numberwithin{equation}{section}

\journal{Journal of \LaTeX\ Templates}

\makeatletter
\def\ps@pprintTitle{%
   \let\@oddhead\@empty
   \let\@evenhead\@empty
   \def\@oddfoot{\reset@font\hfil\thepage\hfil}
   \let\@evenfoot\@oddfoot
}
\makeatother

\begin{document}

\begin{frontmatter}

\title{$P_k$ and $C_k$-structure and substructure connectivity of hypercubes}
%\tnotetext[mytitlenote]{Fully documented templates are available in the elsarticle package on \href{http://www.\;ctan.\;org/tex-archive/macros/latex/contrib/elsarticle}{CTAN}.\;}
\author[]{Yihan Chen\corref{cor1}}
\ead{2016750518@smail.xtu.edu.cn}
\author[]{Bicheng Zhang\corref{cor1}}
\ead{zhangbicheng@xtu.edu.cn}

\address{School of Mathematics and Computational Science, Xiangtan Univerisity, Xiangtan, Hunan, 411105, PR China}
\cortext[cor1]{Corresponding author}

%\fntext[fn1]{Supported by Hunan Provincial Natural Science Foundation of China (09JJ4002).\;}
%\fntext[fn2]{Supported by Natural Science Foundation of China (11471108).\;}
%\fntext[fn2]{Another author footnote,\; this is a very long
%footnote and it should be a really long footnote.\; But this
%footnote is not yet sufficiently long enough to make two lines
%of footnote text.\;}
%\fntext[fn3]{Yet another author footnote.\;}
%\address{Radarweg 29,\; Amsterdam}
%\fntext[myfootnote]{Since 1880.\;}

%\author[mymainaddress,\;mysecondaryaddress]{Elsevier Inc}
%\ead[url]{www.\;elsevier.\;com}

%\ead{support@elsevier.\;com}
%
%\address[mymainaddress]{1600 John F Kennedy Boulevard,\; Philadelphia}
%\address[mysecondaryaddress]{360 Park Avenue South,\; New York}

\begin{abstract}
Hypercube is one of the most important networks to interconnect processors in multiprocessor computer systems. Different kinds of connectivities are important parameters to measure the fault tolerability of networks. Lin et al.\cite{LinStructure} introduced the concept of $H$-structure connectivity $\kappa(Q_n;H)$ (resp. $H$-substructure connectivity $\kappa^s(Q_n;H)$) as the minimum cardinality of $F=\{H_1,\dots,H_m\}$ such that $H_i (i=1,\dots,m)$ is isomorphic to $H$ (resp. $F=\{H'_1,\dots,H'_m\}$ such that $H'_i (i=1,\dots,m)$ is isomorphic to connected subgraphs of $H$) such that $Q_n-V(F)$ is disconnected or trivial. In this paper, we discuss $\kappa(Q_n;H)$ and $\kappa^s(Q_n;H)$ for hypercubes $Q_n$ with $n\geq 3$ and $H\in \{P_k,C_k|3\leq k\leq 2^{n-1}\}$. As a by-product, we solve the problem mentioned in \cite{ManeStructure}.
\end{abstract}
\begin{keyword}
Structure connectivity;\;Substructure connectivity;\;Hypercubes;\;Path;\;Cycle
\end{keyword}
\end{frontmatter}
%\linenumbers
\section{Introduction}
In the design of multiprocessor computer systems, the analysis of topological properties of interconnection networks is very important. Usually, an efficient interconnection topology can offer higher speed and better quality of information transmission between processors, most importantly, it can also provide high fault tolerability.

Many regular networks have been proposed as the model to interconnect processors, such as hypercube, mesh and twiced cube etc. In the past, people used vertex connectivity to measure the fault tolerability of interconnect networks, their works focuse on the effect of individual processors becoming fault. People soon discovered that there were some shortcomings in the using of vertex connectivity, in order to make up for these shortcomings, Harary defined the conditional connectivity. Later, many people generalized vertex connectivity in various ways and got many benefits, for example, restricted connectivity, $g$-extra connectivity and $R_g$-connectivity etc.

Lin et al.\cite{LinStructure} introduced the concept of $H$-structure connectivity $\kappa(Q_n;H)$ and the $H$-substructure connectivity $\kappa^s(Q_n;H)$ and determined them for $H\in\{K_1,K_{1,1},K_{1,2},K_{1,3},C_4\}$. Li et al.\cite{LI2019169} determined $\kappa(H_n;T)$ and $\kappa^s(H_n;T)$ with $H_n$ the twisted hypercube and $T\in\{K_{1,r},P_k|r=3,4 \text{ and } 1\leq k\leq n\}$. Mane \cite{ManeStructure} proved that if $n\geq 4$, then for each $2\leq m \leq n-2$, $\kappa(Q_n;C_{2^m})\leq n-m$ and gave an open problem: For $n\geq 4$ and for each $2\leq m\leq n-2$, can we prove $\kappa(Q_n;C_{2^m})= n-m$?

In this paper, we discuss $\kappa(Q_n;H)$ and $\kappa^s(Q_n;H)$ for hypercubes $Q_n$ with $n\geq 3$ and $H\in \{P_k,C_k|3\leq k\leq 2^{n-1}\}$. As a by-product, we solve the problem above.

The rest of this paper is organized as follows. In \emph{Section 2} we give the fundamental definitions and theorems. In \emph{Section 3} we discuss $\kappa(Q_n;P_k)$ and $\kappa^s(Q_n;P_k)$ for hypercubes $Q_n$ with $n\geq 3$ and $3\leq k\leq 2^{n-1}$. In \emph{Section 4} we discuss $\kappa(Q_n;C_k)$ and $\kappa^s(Q_n;C_k)$ for hypercubes $Q_n$ with $n\geq 3$ and $4\leq k\leq 2^{n-1}$, $k$ is even and solve the open problem above.

\section{Definitions and preliminaries}
We are consistent with \cite{diestel2012graph} in terms of notations and definitions. Let $G(V,E)$ be an undirected simple graph with vertex set $V$ and edge set $E$, two vertices $u$ and $v$ are \emph{adjacent} if and only if they are two endpoints of some edge $uv$ in $E$. Let $S$ be a subset of $V$, we denote $N_G(S)$ the set of neighborhoods of $S$, particularly $N_G(\{u\})=N_G(u)$ is the set of neighborhoods of $u$. In graph theory, a \emph{path} of length $k-1$ is a sequence of distinct vertices $P_k=(v_0,v_1,\dots,v_{k-1})$ in which $v_iv_{i+1}\in E$ for $i=0,1,\dots,k-2$. A \emph{cycle} of length $l$ is a sequence of vertices $C_l=(v_0,v_1,\dots,v_l)$ in which $v_0=v_l$ and $v_i$ are distinct with $v_iv_{i+1}\in E$ for $i=0,1,\dots,l-1$. An undirected graph $G(V,E)$ is said to be \emph{connected} if and only if it has at least one vertex and there is a path between every pair of its vertices.\par
A \emph{$n$-dimensional hypercube} $Q_n$ is an undirected simple graph with vertex set $V(Q_n)=\{v=x_v^0x_v^1\dots x_v^{n-1}\mid x_v^i=0\;or\;1, i=0,\dots,n-1 \}$ and edge set $E(Q_n)$. Two vertices are adjacent if and only if the $n$-bit binary strings corresponding to them are differ in exactly one bit position. Let $v$ be a vertex of $Q_n$, we denote $(v)^i$ the neighborhood of $v$ whose $n$-bit binary string is differ from $v$'s in exactly the $i^{th}$ bit position for $i=0,1,\dots,n-1$. We notice that $Q_n$ has $2^n$ vertices and $n2^{n-1}$ edges, and it can be divided into two $n-1$-dimentional hypercubes. We set $Q_n^i$ the induced subgraph of $Q_n$ with vertex set $V(Q_n^i)=\{v\mid v\in V(Q_n), x_v^{n-1}=i \}$ for $i=0,1$. It is easy to see that $Q_n^i$ is isomorphic to $Q_{n-1}$.\par
In the following, let's introduce some necessary definitions about the \emph{structure and substructure connectivity}, let $H$ and $G$ be undirected simple graph, $H$ is connected.
\begin{dfn}
	A \emph{$H$-structure cut (resp. $H$-substructure cut)} of $G$ is a set $F$ of connected subgraphs of $G$ whose elements are isomorphic to $H$ \emph{(resp. connected subgraph of $H$)} and $G-V(F)$ is trivial or disconnected.
\end{dfn}
See {\bf Fig 1} and {\bf Fig 2} for an illustration.
\begin{figure}[h] 
	\centering
	\def\svgwidth{300px}
	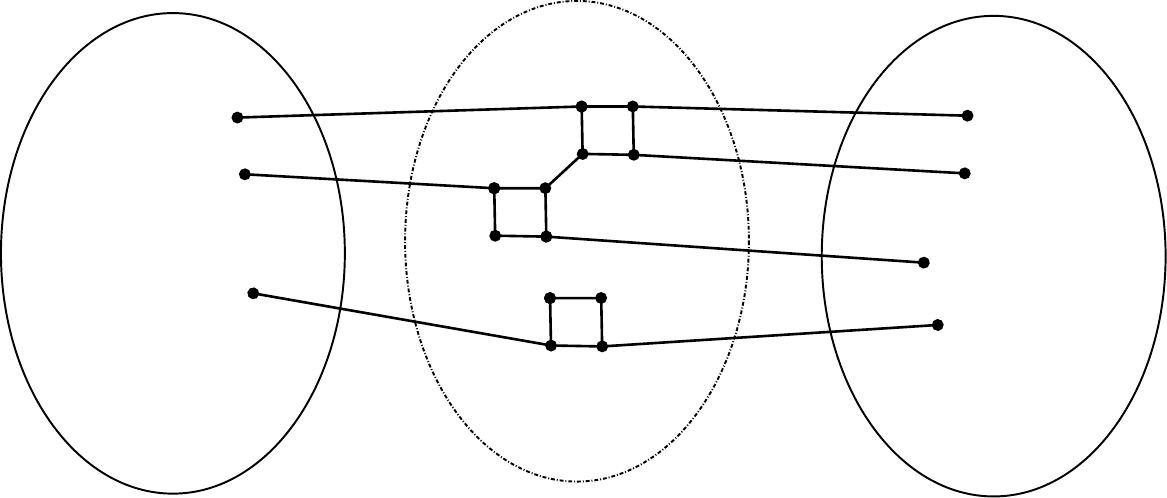
\end{figure}
\begin{center}
	\textbf{Fig 1.} A $C_4$-structure cut
\end{center}
\begin{figure}[h] 
	\centering
	\def\svgwidth{300px}
	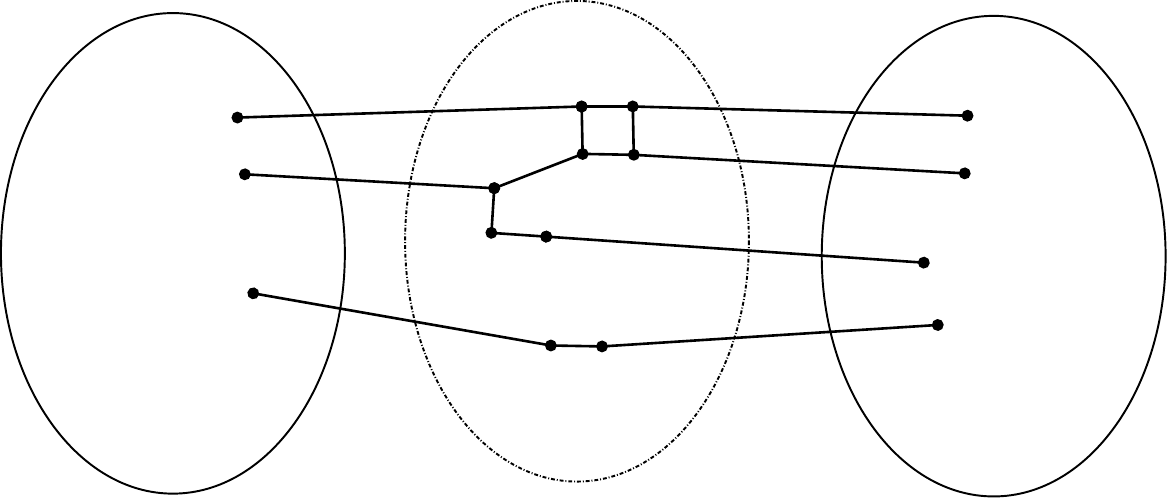
\end{figure}
\begin{center}
	\textbf{Fig 2.} A $C_4$-substructure cut
\end{center}
\begin{dfn}
	The \emph{$H$-structure connectivity (resp. $H$-substructure connectivity)} denote by $\kappa(G;H)$ (\emph{resp.} $\kappa^s(G;H)$) is the minimum cardinality of $H$-structure cuts (\emph{resp.} $H$-substructure cuts).
\end{dfn}
\begin{dfn}(See {\cite{ZHOU2017208}})
	Given a non-negative integer $g$, \emph{$g$-extra connectivity} of $G$ denote by $\kappa_g(G)$ is the minimum cardinality of sets of vertices in $G$ whose deletion disconnects $G$ and leaves each components with at least $g+1$ vertices.
\end{dfn}
\begin{lem}\label{dist2nodes}
	Any two vertices in $V(Q_n)$ with a distance of $2$ have exactly $2$ common neighborhoods.
\end{lem}
\begin{proof}
	According to the symmetry of the hypercube $Q_n$, without loss of generalicity we assume that the two vertices are $u=000\dots0$ and $v=110\dots0$, obviously their common neighborhoods are $u'=100\dots0$ and $v'=010\dots0$.
\end{proof}
\begin{thm}\label{gextra}(See {\cite{YangExtraconnectivity}})
	If $n\geq 4$, then $\kappa_g(Q_n)=(g+1)n-2g-\binom{g}{2}$ for $0\leq g\leq n-4$, and $\kappa_g(Q_n)=\frac{n(n-1)}{2}$ for $n-3\leq g\leq n$.
\end{thm}
\begin{thm}\label{linsresult}(See {\cite{LinStructure}})
	For $n\geq 4$
	\begin{equation}
	\label{res1}
	\kappa(Q_n;H)=
	\begin{cases}
	n & \text{$H=K_1$,} \\
	n-1 & \text{$H=K_{1,1}$,} \\
	\lceil \frac{n}{2} \rceil & \text{$H\in \{K_{1,2},K_{1,3}\}$,} \\
	n-2 & \text{$H=C_4$.}
	\end{cases}
	\end{equation}
	\begin{equation}
	\label{res2}
	\kappa^s(Q_n;H)=
	\begin{cases}
	n & \text{$H=K_1$,} \\
	n-1 & \text{$H=K_{1,1}$,} \\
	\lceil \frac{n}{2} \rceil & \text{$H\in \{K_{1,2},K_{1,3},C_4\}$.}
	\end{cases}
	\end{equation}
\end{thm}
\begin{thm}\label{Manesresult}(See \cite{ManeStructure})
	if $n\geq 4$, then for each $2\leq m \leq n-2$, $\kappa(Q_n;C_{2^m})\leq n-m$.
\end{thm}
\begin{lem}\label{TAMI}(See \cite{TAMIZHCHELVAM2019})
	For an integer $n\geq 4$, $\kappa(Q_n;C_6)\geq \lceil \frac{n}{3}\rceil$.
\end{lem}
\begin{thm}\label{hamilton}(See {\cite{SaadTopological}})
	$Q_n$ is Hamiltonian.
\end{thm}
\section{$\kappa(Q_n;P_k)$ and $\kappa^s(Q_n;P_k)$ for $n\geq 3$ and $3\leq k \leq 2^{n-1}$}

In this section we discuss $H$-structure connectivity and $H$-substructure connectivity of $Q_n$ with $H$ the path of length $k-1$.
\begin{lem}\label{pathleq}
	If $n\geq 3$ and $3\leq k \leq 2^{n-1}$, then 
	\begin{equation}    \kappa^s(Q_n;P_k)\leq \kappa(Q_n;P_k)\leq
	\begin{cases}
	\lceil \frac{2n}{k+1} \rceil   &  \text{if $k$ is odd,} \\
	\lceil \frac{2n}{k} \rceil &  \text{if $k$ is even.}
	\end{cases}                
	\end{equation}
\end{lem}
\begin{proof}
	We set $v=00\dots0$ being a vertex in $Q_n$.
\vskip .2cm
{\bf Case 1.} $k$ is odd.
\vskip .2cm
{\bf Subcase 1.1.} $k\geq 2n-1$.
We notice that there is a hamiltonian cycle in $Q_n^1$ with ${({(v)}^{n-2})}^{n-1}(v)^{n-1}$ as one of its edges by Theorem \ref{hamilton} and the symmetry of $Q_n^1$. For each $i=1,2,\dots,2^{n-1}-2$, set $u_i$ the $i^{th}$ vertex after ${\left({\left(v\right)}^{n-2}\right)}^{n-1}$ and $(v)^{n-1}$ along the hamiltonian cycle. Then $F=\left\{\left((v)^0,{\left(\left(v\right)^0\right)}^1,(v)^1,\dots,{(v)}^{n-2},{\left({\left(v\right)}^{n-2}\right)}^{n-1},(v)^{n-1},u_1,\dots,u_{k-(2n-1)}\right)\right\}$ is a $P_k$-structure cut of $Q_n$ leaves at least two connected components in $Q_n-V(F)$ and $\{v\}$ is one of them. Therefore $\kappa(Q_n;P_k)\leq 1=\lceil \frac{2n}{k+1} \rceil$.
\begin{figure}[h] 
	\centering
	\def\svgwidth{300px}
%% Accompanies image file 'pathleq1.pdf' (pdf, eps, ps)
\begingroup%
  \makeatletter%
  \providecommand\color[2][]{%
    \errmessage{(Inkscape) Color is used for the text in Inkscape, but the package 'color.sty' is not loaded}%
    \renewcommand\color[2][]{}%
  }%
  \providecommand\transparent[1]{%
    \errmessage{(Inkscape) Transparency is used (non-zero) for the text in Inkscape, but the package 'transparent.sty' is not loaded}%
    \renewcommand\transparent[1]{}%
  }%
  \providecommand\rotatebox[2]{#2}%
  \newcommand*\fsize{\dimexpr\f@size pt\relax}%
  \newcommand*\lineheight[1]{\fontsize{\fsize}{#1\fsize}\selectfont}%
  \ifx\svgwidth\undefined%
    \setlength{\unitlength}{299.28257439bp}%
    \ifx\svgscale\undefined%
      \relax%
    \else%
      \setlength{\unitlength}{\unitlength * \real{\svgscale}}%
    \fi%
  \else%
    \setlength{\unitlength}{\svgwidth}%
  \fi%
  \global\let\svgwidth\undefined%
  \global\let\svgscale\undefined%
  \makeatother%
  \begin{picture}(1,0.44653062)%
    \lineheight{1}%
    \setlength\tabcolsep{0pt}%
    \put(0,0){\includegraphics[width=\unitlength,page=1]{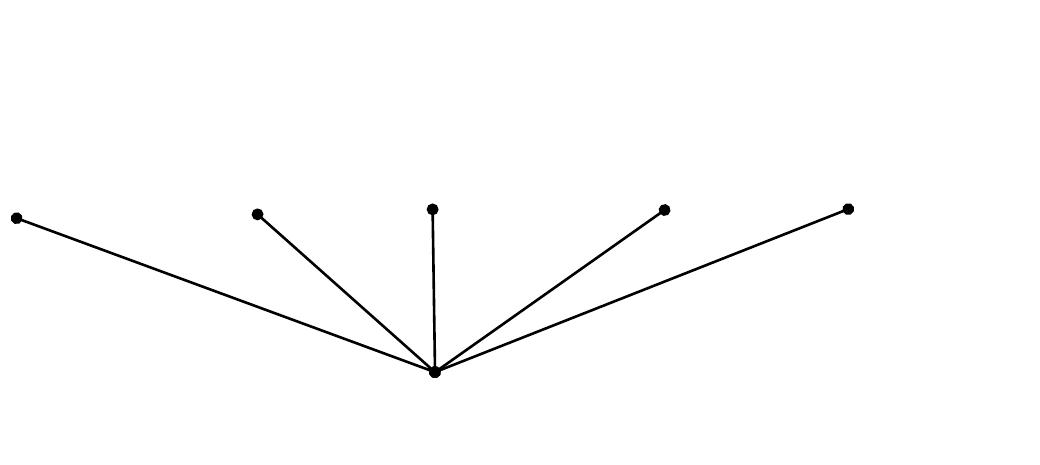}}%
    \put(0.39260019,0.04892125){\color[rgb]{0,0,0}\makebox(0,0)[lt]{\lineheight{1.25}\smash{\begin{tabular}[t]{l}$v$\end{tabular}}}}%
    \put(0.52684991,0.25879806){\color[rgb]{0,0,0}\makebox(0,0)[lt]{\lineheight{1.25}\smash{\begin{tabular}[t]{l}$\dots$\end{tabular}}}}%
    \put(0,0){\includegraphics[width=\unitlength,page=2]{pathleq1.pdf}}%
    \put(0.04578868,0.06637372){\color[rgb]{0,0,0}\makebox(0,0)[lt]{\lineheight{1.25}\smash{\begin{tabular}[t]{l}$Q_n^0$\end{tabular}}}}%
    \put(0.7738863,0.14718796){\color[rgb]{0,0,0}\makebox(0,0)[lt]{\lineheight{1.25}\smash{\begin{tabular}[t]{l}$Q_n^1$\end{tabular}}}}%
  \end{picture}%
\endgroup%

\end{figure}
\begin{center}
	 
\end{center}
\begin{center}
	\textbf{Fig 3.} $k\geq 2n-1$
\end{center}
\vskip .2cm

{\bf Subcase 1.2.} $1\leq k\leq 2n-1$.
\vskip .2cm
{\bf Subcase 1.2.1.}$\frac{k+1}{2}$ divides $n$.
We set $F=\{P_k^i|i=0,1,\dots,\frac{2n}{k+1}-1\}$ with \\ $P_k^i=\left((v)^{i\frac{k+1}{2}},{\left(\left(v\right)^{i\frac{k+1}{2}}\right)}^{i\frac{k+1}{2}+1},(v)^{i\frac{k+1}{2}+1},\dots,(v)^{(i+1)\frac{k+1}{2}-1}\right)$. Then $F$ is obviously a $P_k$-structure cut of $Q_n$, implies $\kappa(Q_n;P_k)\leq \frac{2n}{k+1}=\lceil \frac{2n}{k+1} \rceil$.

\vskip .2cm
{\bf Subcase 1.2.2.}$\frac{k+1}{2}$ does not divide $n$.
We set $F=\{P_k^i|i=0,1,\dots,\lceil \frac{2n}{k+1} \rceil-1\}$ with\\ $P_k^i=\left((v)^{i\frac{k+1}{2}},{\left(\left(v\right)^{i\frac{k+1}{2}}\right)}^{i\frac{k+1}{2}+1},(v)^{i\frac{k+1}{2}+1},\dots,(v)^{(i+1)\frac{k+1}{2}-1}\right)$ for $i=0,\dots,\lceil \frac{2n}{k+1} \rceil-2$ and $P_k^{\lceil \frac{2n}{k+1} \rceil-1}=\left((v)^{n-\frac{k+1}{2}},{\left(\left(v\right)^{n-\frac{k+1}{2}}\right)}^{n-\frac{k+1}{2}+1},(v)^{n-\frac{k+1}{2}+1},\dots,(v)^{n-1}\right)$. Then $F$ is obviously a $P_k$-structure cut of $Q_n$, implies $\kappa(Q_n;P_k)\leq \lceil \frac{2n}{k+1} \rceil$.\\
Let's have $n=5$ and $k=3$ as an example:
\begin{figure}[h] 
	\centering
	\def\svgwidth{252px}
%% Accompanies image file 'pathleq2.pdf' (pdf, eps, ps)
\begingroup%
  \makeatletter%
  \providecommand\color[2][]{%
    \errmessage{(Inkscape) Color is used for the text in Inkscape, but the package 'color.sty' is not loaded}%
    \renewcommand\color[2][]{}%
  }%
  \providecommand\transparent[1]{%
    \errmessage{(Inkscape) Transparency is used (non-zero) for the text in Inkscape, but the package 'transparent.sty' is not loaded}%
    \renewcommand\transparent[1]{}%
  }%
  \providecommand\rotatebox[2]{#2}%
  \newcommand*\fsize{\dimexpr\f@size pt\relax}%
  \newcommand*\lineheight[1]{\fontsize{\fsize}{#1\fsize}\selectfont}%
  \ifx\svgwidth\undefined%
    \setlength{\unitlength}{214.63392827bp}%
    \ifx\svgscale\undefined%
      \relax%
    \else%
      \setlength{\unitlength}{\unitlength * \real{\svgscale}}%
    \fi%
  \else%
    \setlength{\unitlength}{\svgwidth}%
  \fi%
  \global\let\svgwidth\undefined%
  \global\let\svgscale\undefined%
  \makeatother%
  \begin{picture}(1,0.51828283)%
    \lineheight{1}%
    \setlength\tabcolsep{0pt}%
    \put(0,0){\includegraphics[width=\unitlength,page=1]{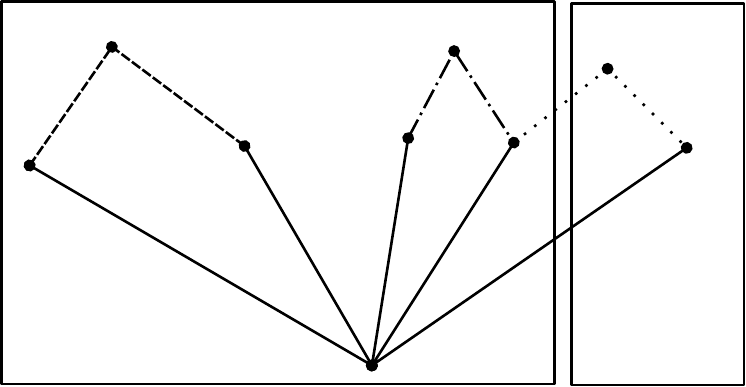}}%
    \put(0.51903156,0.01547508){\color[rgb]{0,0,0}\makebox(0,0)[lt]{\lineheight{1.25}\smash{\begin{tabular}[t]{l}$v$\end{tabular}}}}%
    \put(0.04792214,0.05229016){\color[rgb]{0,0,0}\makebox(0,0)[lt]{\lineheight{1.25}\smash{\begin{tabular}[t]{l}$Q_5^0$\end{tabular}}}}%
    \put(0.7848496,0.04417813){\color[rgb]{0,0,0}\makebox(0,0)[lt]{\lineheight{1.25}\smash{\begin{tabular}[t]{l}$Q_5^1$\end{tabular}}}}%
  \end{picture}%
\endgroup%

\end{figure}

\begin{center}
	\textbf{Fig 4.} $n=5$ and $k=3$
\end{center}
\vskip .2cm

{\bf Case 2.} $k$ is even.
\vskip .2cm
{\bf Subcase 2.1.} $k\geq 2n$. The proof is as same as {\bf Subcase 1.1.}
\vskip .2cm
{\bf Subcase 2.2.} $2\leq k\leq 2n-2$.
\vskip .2cm
{\bf Subcase 2.2.1.} $\frac{k}{2}$ divides $n$.
We set $F=\{P_k^i|i=0,1,\dots,\frac{2n}{k}-1\}$ with \\ $P_k^i=\left((v)^{i\frac{k}{2}},{\left(\left(v\right)^{i\frac{k}{2}}\right)}^{i\frac{k}{2}+1},(v)^{i\frac{k}{2}+1},\dots,(v)^{(i+1)\frac{k}{2}-1},{\left(\left(v\right)^{(i+1)\frac{k}{2}-1}\right)}^{(i+1)\frac{k}{2}}\right)$ for $i=0,\dots, \frac{2n}{k}-2$ and $P_k^{\frac{2n}{k}-1}=\left((v)^{\left(\frac{2n}{k}-1\right)\frac{k}{2}},{\left(\left(v\right)^{\left(\frac{2n}{k}-1\right)\frac{k}{2}}\right)}^{\left(\frac{2n}{k}-1\right)\frac{k}{2}+1},(v)^{\left(\frac{2n}{k}-1\right)\frac{k}{2}+1},\dots,(v)^{n-1},{\left(\left(v\right)^{n-1}\right)}^{0}\right)$.  Then $F$ is obviously a $P_k$-structure cut of $Q_n$, implies $\kappa(Q_n;P_k)\leq \frac{2n}{k}=\lceil \frac{2n}{k} \rceil$.

\vskip .2cm
{\bf Subcase 2.2.2.} $\frac{k}{2}$ does not divide $n$.
We set $F=\{P_k^i|i=0,1,\dots,\lceil \frac{2n}{k} \rceil-1\}$ with \\ $P_k^i=\left((v)^{i\frac{k}{2}},{\left(\left(v\right)^{i\frac{k}{2}}\right)}^{i\frac{k}{2}+1},(v)^{i\frac{k}{2}+1},\dots,(v)^{(i+1)\frac{k}{2}-1},{\left(\left(v\right)^{(i+1)\frac{k}{2}-1}\right)}^{(i+1)\frac{k}{2}}\right)$ for $i=0,\dots,\lceil \frac{2n}{k} \rceil-2$ and $P_k^{\lceil \frac{2n}{k} \rceil-1}=\left((v)^{n-\frac{k}{2}},{\left(\left(v\right)^{n-\frac{k}{2}}\right)}^{n-\frac{k}{2}+1},(v)^{n-\frac{k}{2}+1},\dots,(v)^{n-1},{\left(\left(v\right)^{n-1}\right)}^0\right)$. Then $F$ is obviously a $P_k$-structure cut of $Q_n$, implies $\kappa(Q_n;P_k)\leq \lceil \frac{2n}{k} \rceil$.
\end{proof}

\begin{lem}\label{forleqk-1}
	For $n\geq 3$ and $3\leq k \leq 2^{n-1}$, let $k=3q_k+r_k$ with non-negative integers $q_k$ and $r_k$, $0\leq r_k\leq 2$. If $u$ and $v$ are adjacent in $Q_n-V(P_k)$, then $|N_{Q_n}(\{u,v\})\cap V(P_k)|\leq 2q_k+r_k$. In particular,  $|N_{Q_n}(\{u,v\})\cap V(P_k)|\leq k-1$.
\end{lem}
\begin{proof}
	Set the vertices in $V(P_k)$ in turn as $v_i$ for $i=1,2,\dots,k$. First of all, we prove $|N_{Q_n}(\{u,v\})\cap \{v_{i-1},v_i,v_{i+1}\}|\leq 2$ for $i=2,3,\dots,k-1$ by contradiction. Suppose $|N_{Q_n}(\{u,v\})\cap \{v_{i'-1},v_{i'},v_{i'+1}\}|=3$ for some integer $2\leq i'\leq k-1$, without loss of generality, we assume $u$ is adjacent to $v_{i'-1}$ and $v_{i'+1}$ and $v$ is adjacent to $v_{i'}$. Then $d(u,v_{i'})=2$ but $u$ and $v_{i'}$ have $3$ common neighborhoods: $v$,$v_{i'-1}$,$v_{i'+1}$, leads a contradiction by Lemma \ref{dist2nodes}.

	Then we prove the lemma by induction on $k$. For $k=3$, $|N_{Q_n}(\{u,v\})\cap \{v_{1},v_2,v_{3}\}|\leq 2$ by the statement above. Assume that $|N_{Q_n}(\{u,v\})\cap V(P_k)|\leq 2q_k+r_k$ is true for $3\leq k \leq l$ with integer $3\leq l \leq 2^{n-1}-1$. For $k=l+1$, $|N_{Q_n}(\{u,v\})\cap V(P_{l+1})|\leq |N_{Q_n}(\{u,v\})\cap V(P_{l})|+1\leq 2q_l+r_l+1$ by inductive hypothesis.
	\vskip .2cm
	{\bf Case 1.} $r_{l+1}=1$ and $2$. Notice that in this case $r_{l+1}=r_{l}+1$ and $q_{l+1}=q_{l}$, so we have $|N_{Q_n}(\{u,v\})\cap V(P_{l+1})|\leq 2q_l+r_l+1=2q_{l+1}+r_{l+1}$.
	\vskip .2cm
	{\bf Case 2.} $r_{l+1}=0$ and $l\geq 5$. Since $|N_{Q_n}(\{u,v\})\cap \{v_{i-1},v_i,v_{i+1}\}|\leq 2 (i=2,3,\dots,k-1)$ we have $|N_{Q_n}(\{u,v\})\cap V(P_{l+1})|\leq |N_{Q_n}(\{u,v\})\cap V(P_{l-2})|+2\leq 2q_{l-2}+2\leq 2q_{l+1}$.\\
	In particular, $|N_{Q_n}(\{u,v\})\cap V(P_k)|\leq 2q_k+r_k=k-q_k\leq k-1$.
	
\end{proof}

\begin{lem}\label{pathgeq}
	For $n\geq 3$ and $3\leq k\leq 2^{n-1}$, 
	\begin{equation}    \kappa^s(Q_n;P_k)\geq
	\begin{cases}
	\lceil \frac{2n}{k+1} \rceil   &  \text{if $k$ is odd,} \\
	\lceil \frac{2n}{k} \rceil &  \text{if $k$ is even.}
	\end{cases}                
	\end{equation}
\end{lem}
\begin{proof}
	Let $F=\{P_1^{1},P_1^{2},\dots,P_1^{a_1},P_2^{1},\dots,P_2^{a_2},\dots,P_k^{1},\dots,P_k^{a_k}\}$ be a set of paths in $Q_n$ of length less than $k-1$.
	
	For $n=3$, $k$ must be $3$ or $4$. Suppose to the contrary that $Q_3-V(F)$ is disconnected with $|F|\leq \lceil \frac{6}{k} \rceil -1=1$ $(k=3,4)$. If $k=3$, then $Q_3-V(F)$ is connected by Theorem \ref{linsresult}. If $k=4$ and $a_4=0$, then $Q_3-V(F)$ is connected by Theorem \ref{linsresult}. If $a_4=1$, let $F=\{P_4\}$ and all possible cases of $P_4$ are shown in {\bf Fig 5} by the symmetry of $Q_3$, and we can see $Q_3-V(F)$ is connected.
	\begin{center}
		\begin{figure}[H] 
			\centering
			\def\svgwidth{250px}
			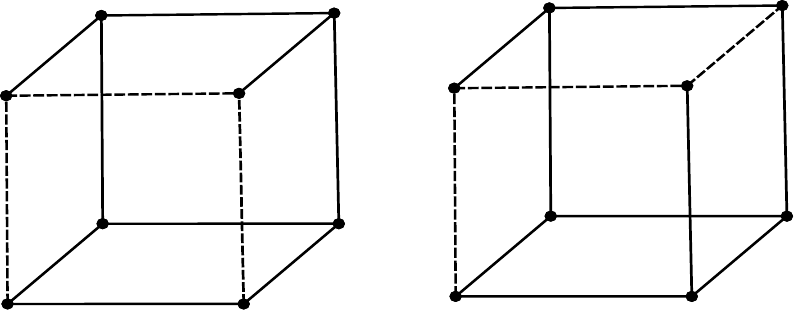
		\end{figure}
		\begin{center}
			\textbf{Fig 5.} 
		\end{center}
	\end{center}
	
	For $n\geq 4$, we have the following cases.

	{\bf Case 1.} $k$ is odd.
	We prove that for $|F|\leq \lceil \frac{2n}{k+1} \rceil-1$, $Q_n-V(F)$ is connected by contradiction. Suppose $Q_n-V(F)$ is disconnected, let $C$ be one of the smallest components of it, that means components who have the smallest number of vertices.
	\vskip .2cm
	{\bf Subcase 1.1.} $|V(C)|=1$. Let $V(C)=\{u\}$, each path in $F$ contains at most $\frac{k+1}{2}$ neighbors of $u$ since there is no $3$-cycle in $Q_n$. Therefore $\frac{k+1}{2}|F|\geq n$ i.e. $\frac{k+1}{2}(\lceil \frac{2n}{k+1} \rceil-1)\geq n$, a contradiction with $\frac{k+1}{2}(\lceil \frac{2n}{k+1} \rceil-1)< n$.
	
	\vskip .2cm
	{\bf Subcase 1.2.} $|V(C)|\geq 2$. 
	For $n=4$, we have $|V(F)|\leq k(\lceil \frac{8}{k+1} \rceil-1)<6$, a contradiction with $\kappa_1(Q_4)=\frac{4(4-1)}{2}=6$ by Theorem \ref{gextra}. For $n\geq 5$, we have
 $\kappa_1(Q_n)=2n-2$ by Theorem \ref{gextra}, the number of deleted vertices should be at least $2n-2$ i.e. $|V(F)|\geq 2n-2$. But $|V(F)|\leq k(\lceil \frac{2n}{k+1} \rceil-1)\leq k(\frac{2n+k-1}{k+1}-1)=\frac{k}{k+1}(2n-2)<2n-2$, a contradiction.
	\vskip .2cm
	{\bf Case 2.} $k$ is even.
	We prove that for $|F|\leq \lceil \frac{2n}{k} \rceil-1$, $Q_n-V(F)$ is connected by contradiction. Suppose $Q_n-V(F)$ is disconnected, let $C$ be one of the smallest components of it.
	
	\vskip .2cm
	{\bf Subcase 2.1.} $|V(C)|=1$. 
	Let $V(C)=\{u\}$, each path in $F$ contains at most $\frac{k}{2}$ neighbors of $u$. Therefore $\frac{k}{2}|F|\geq n$ i.e. $\frac{k}{2}(\lceil \frac{2n}{k} \rceil-1)\geq n$, a contradiction with $\frac{k}{2}(\lceil \frac{2n}{k} \rceil-1)< n$.
	
	\vskip .2cm
	{\bf Subcase 2.2.} $|V(C)|=2$. Let $V(C)=\{u,v\}$, it is obvious that $|N_{Q_n}(\{u,v\})|=2n-2$. For each path $P \in F$, we have $|N_{Q_n}(\{u,v\})\cap V(P)|\leq k-1$ by Lemma \ref{forleqk-1}. Therefore $F$ must has at least $\lceil \frac{2n-2}{k-1} \rceil$ elements, but $|F|\leq \lceil \frac{2n}{k} \rceil -1\leq \frac{2n+k-2}{k}-1=\frac{2n-2}{k}<\lceil \frac{2n-2}{k-1} \rceil$, a contradiction.

	\vskip .2cm
	{\bf Subcase 2.3.} $|V(C)|\geq 3$. 
	
	\vskip .2cm
	{\bf Subcase 2.3.1.} For $n\geq 6$, we have $|V(F)|\leq k(\lceil \frac{2n}{k} \rceil-1)<3n-5=\kappa_2(Q_n)$, a contradiction.

	\vskip .2cm
	{\bf Subcase 2.3.2.}
	For $n=5$, we have $|V(F)|\leq k(\lceil \frac{10}{k} \rceil-1)<\frac{5(5-1)}{2}=10=\kappa_2(Q_5)$, a contradiction.

	\vskip .2cm
	{\bf Subcase 2.3.3.}
	For $n=4$ and $k=4,8$, we have $|V(F)|\leq k(\lceil \frac{8}{k} \rceil-1)<\frac{4(4-1)}{2}=6=\kappa_2(Q_4)$, a contradiction.
\begin{table}[H]
	\centering
	\begin{tabular}{|c|c|c|cccc}
		\cline{1-3} \cline{5-7}
		$k$ &
		$k(\lceil \frac{10}{k} \rceil-1)$ &
		$\kappa_2(Q_5)$ &
		\multicolumn{1}{c|}{} &
		\multicolumn{1}{c|}{$k$} &
		\multicolumn{1}{c|}{$k(\lceil \frac{8}{k} \rceil-1)$} &
		\multicolumn{1}{c|}{$\kappa_2(Q_4)$} \\ \cline{1-3} \cline{5-7} 
		$4$ &
		$8$ &
		\multirow{4}{*}{$10$} &
		\multicolumn{1}{c|}{} &
		\multicolumn{1}{c|}{$4$} &
		\multicolumn{1}{c|}{$4$} &
		\multicolumn{1}{c|}{\multirow{2}{*}{$6$}} \\ \cline{1-2} \cline{5-6}
		$6$       & $6$ &  & \multicolumn{1}{c|}{} & \multicolumn{1}{c|}{$8$} & \multicolumn{1}{c|}{$0$} & \multicolumn{1}{c|}{} \\ \cline{1-2} \cline{5-7} 
		$8$       & $8$ &  &                       &                          &                          &                       \\ \cline{1-2}
		$\geq 10$ & $0$ &  &                       &                          &                          &                       \\ \cline{1-3}
	\end{tabular}
\end{table}

For $k=6$, we have $\lceil \frac{8}{6} \rceil-1=1$. If $a_6=0$, we can think of $F$ as a $P_5$-substructure cut and $Q_4-V(F)$ is connected by {\bf Case 1}, a contradiction.

If $a_6=1$, let $F=\{P_6\}$, without loss of generality, we assume that the first $3$ vertices of $P_6$ are $u=0000$, $v=1000$ and $w=1100$ by the symmetry of $Q_4$. If $|V(P_6)\cap V(Q_4^1)|=0$, then $Q_4^1-V(F)$ is connected, therefore $Q_4-V(F)$ is connected since there is a perfect matching between $Q_4^0$ and $Q_4^1$, a contradiction. If $|V(P_6)\cap V(Q_4^1)|=1$, let $V(P_6)\cap V(Q_4^1)=\{x\}$, then $(x)^3\in V(P_6)$ and obviously $Q_4^1-\{x\}$ is connected, therefore $Q_4-V(F)$ is connected by the perfect matching, a contradiction. If $|V(P_6)\cap V(Q_4^1)|=2$, all possible cases of $P_6$ are shown in {\bf Fig 6}, and we can see $Q_4-V(F)$ is connected, a contradiction.
\begin{center}
	\begin{figure}[H] 
		\centering
		\def\svgwidth{350px}
		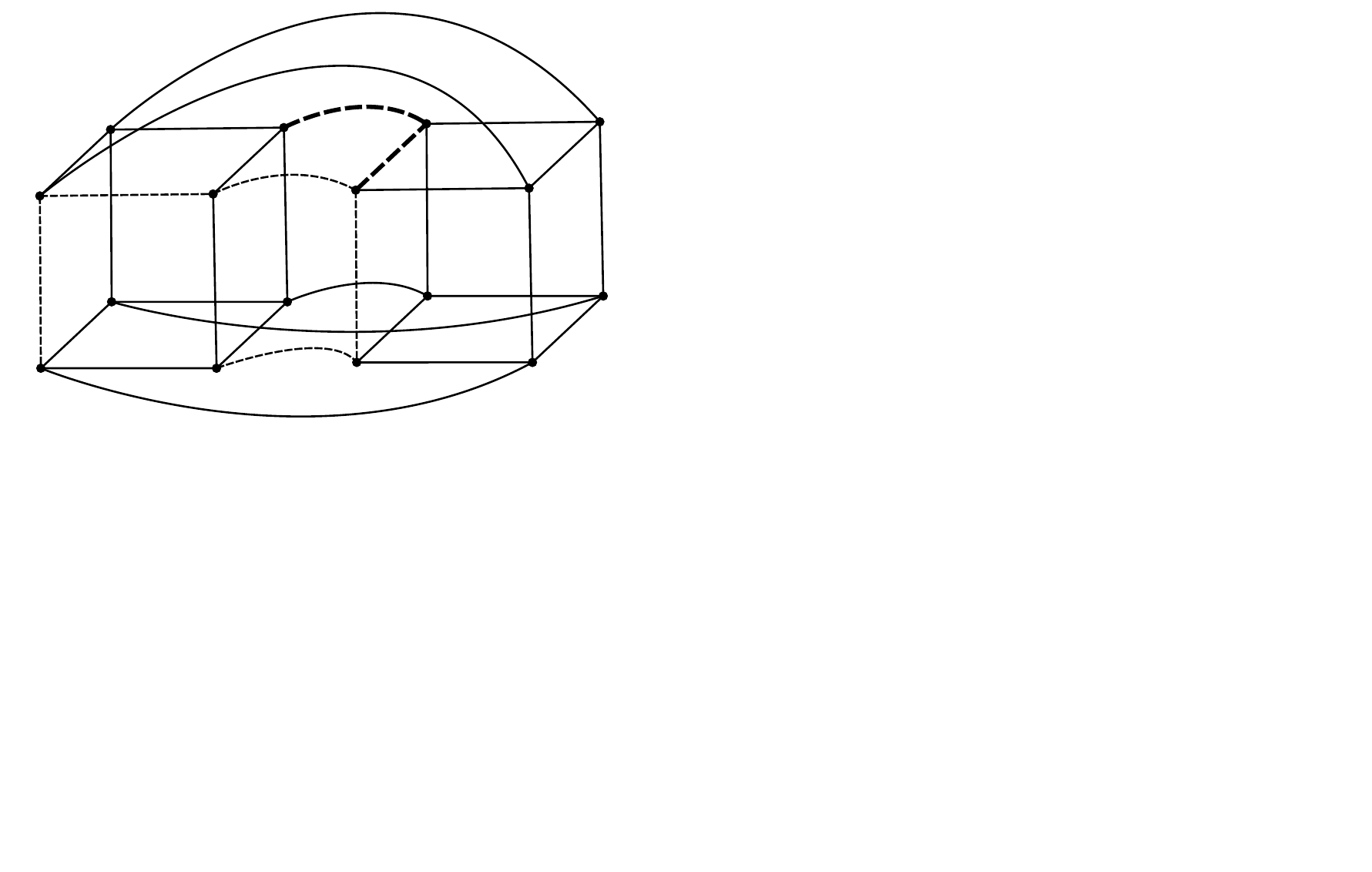
	\end{figure}
	\begin{center}
		\textbf{Fig 6.} Different types of lines represent different ways to construct $P_6$ with startpoint u.
	\end{center}
\end{center}

If $|V(P_6)\cap V(Q_4^1)|=3$, by Theorem \ref{linsresult}, $Q_4^i-V(F)$ is connected for $i=0,1$. Suppose $Q_4-V(F)$ is disconnected, then we have $(y)^3 \in V(F)$ for each $y \in Q_4^0-V(F)$ implies $2^3-3\leq 3$, a contradiction. Therefore $Q_4-V(F)$ is connected, a contradiction.
	
\end{proof}

By Lemma \ref{pathleq} and Lemma \ref{pathgeq}, we have the following theorem.
\begin{thm}\label{paththm}
	For $n\geq 3$ and $3\leq k\leq 2^{n-1}$
	\begin{equation}    \kappa^s(Q_n;P_k)= \kappa(Q_n;P_k)=
	\begin{cases}
	\lceil \frac{2n}{k+1} \rceil   &  \text{if $k$ is odd,} \\
	\lceil \frac{2n}{k} \rceil &  \text{if $k$ is even.}
	\end{cases}                
	\end{equation}
\end{thm}
\section{$\kappa(Q_n;C_k)$ and $\kappa^s(Q_n;C_k)$ for $n\geq 3$ and $4\leq k \leq 2^{n-1}$ and $k$ is even}

In this section we discuss $H$-structure connectivity and $H$-substructure connectivity of $Q_n$ with $H$ the cycle of length $k$.

\begin{lem}\label{cycleoddks}
	For $n\geq 3$ and odd integer $3\leq k\leq 2^{n-1}$, $\kappa^s(Q_n;C_k)=\lceil \frac{2n}{k+1} \rceil$.
\end{lem}
\begin{proof}
	$\kappa^s(Q_n;C_k)=\kappa^s(Q_n;P_k)=\lceil \frac{2n}{k+1} \rceil$ by Theorem \ref{paththm} since there is no odd cycle in $Q_n$.
\end{proof}
\begin{lem}\label{cycleevenks}
	For $n\geq 3$ and even integer $3\leq k\leq 2^{n-1}$, $\kappa^s(Q_n;C_k)\leq\lceil \frac{2n}{k} \rceil$.
\end{lem}
\begin{proof}
	$\kappa^s(Q_n;C_k)\leq \kappa^s(Q_n;P_k)=\lceil \frac{2n}{k} \rceil$ by Theorem \ref{paththm}.
\end{proof}

\begin{lem}\label{cycleleqk-1}
	For $n\geq 3$ and $3\leq k\leq 2^{n-1}$, if vertices $u$ and $v$ are adjacent in $Q_n-V(C_k)$, then $|N_{Q_n}(\{u,v\})\cap V(C_k)|\leq k-1$.
\end{lem}
\begin{proof}
	We prove this lemma by contradiction. Suppose that $|N_{Q_n}(\{u,v\})\cap V(C_k)|=k$ i.e. each vertex in $V(C_k)$ is adjacent to either $u$ or $v$. Take vertices $v_i\in V(C_k)$ for $i=1,2,3$, such that $v_1v_2\in E(Q_n)$ and $v_2,v_3\in E(Q_n)$. Without loss of generality, we assume that $v_1,v_3\in N_{Q_n}(u)$ and $v_2\in N_{Q_n}(v)$, then $d(u,v_2)=2$ but $N_{Q_n}(u)\cap N_{Q_n}(v_2)=\{v,v_1,v_3\}$, a contradiction.
\end{proof}
\begin{lem}\label{cyclesubstr}
	For $n\geq 3$ and even integer $4\leq k \leq 2^{n-1}$, $\kappa^s(Q_n;C_k)\geq \lceil\frac{2n}{k} \rceil$.
\end{lem}
\begin{proof}
	Let $F=\{P_1^{1},P_1^{2},\dots,P_1^{a_1},P_2^{1},\dots,P_2^{a_2},\dots,P_k^{1},\dots,P_k^{a_k},C_k^{1},\dots,C_k^{b}\}$ be a set of paths of length less than $k-1$ and cycles of length $k$. We prove that $Q_n-V(F)$ is connected for $|F|\leq \lceil\frac{2n}{k} \rceil-1$ .
	\vskip .2cm
	{\bf Case 1.} $b=0$. In this case, we can think of $F$ a $P_k$-substructure cut, $Q_n-V(F)$ is connected by Theorem \ref{paththm}.
	
	\vskip .2cm
	{\bf Case 2.} $b>0$.
	
	\vskip .2cm
	{\bf Subcase 2.1.} $n=3$. $k$ must be $4$ and $|F|\leq \lceil\frac{6}{4} \rceil-1=1$. Let $F=\{C_4\}$ and $V(C_4)=\{v_0=000,v_1=100,v_2=110,v_3=010\}$ by the symmetry of $Q_3$, it is clear that $Q_3-V(F)$ is connected.
	
	\vskip .2cm
	{\bf Subcase 2.2.} $n\geq 4$. We prove that $Q_n-V(F)$ is connected by contradiction. Suppose $Q_n-V(F)$ is disconnected, let $C$ be one of the smallest components of it.
	
	\vskip .2cm
	{\bf Subcase 2.2.1.} $|V(C)|=1$.
	Let $V(C)=\{u\}$, each element in $F$ contains at most $\frac{k}{2}$ neighbors of $u$. Therefore $\frac{k}{2}|F|\geq n$ i.e. $\frac{k}{2}(\lceil \frac{2n}{k} \rceil-1)\geq n$, a contradiction.
	\vskip .2cm
	{\bf Subcase 2.2.2.} $|V(C)|=2$. 
	Let $V(C)=\{u,v\}$, it is obvious that $|N_{Q_n}(\{u,v\})|=2n-2$. For each element $G \in F$, we have $|N_{Q_n}(\{u,v\})\cap V(G)|\leq k-1$ by Lemma \ref{forleqk-1} and Lemma \ref{cycleleqk-1}. Therefore $F$ must has at least $\lceil \frac{2n-2}{k-1} \rceil$ elements, but $|F|\leq \lceil \frac{2n}{k} \rceil -1\leq \frac{2n+k-2}{k}-1=\frac{2n-2}{k}<\lceil \frac{2n-2}{k-1} \rceil$, a contradiction.
	
	\vskip .2cm
	{\bf Subcase 2.2.3.} $|V(C)|\geq 3$.
	
	\vskip .2cm
	{\bf Subcase 2.2.3.1.} For $n\geq 5$ and $n=4, k=4,8$ the discussion is the same as Lemma \ref{pathgeq}.
	
	\vskip .2cm
	{\bf Subcase 2.2.3.2.} For $n=4, k=6$, $|F|\leq \lceil\frac{8}{6} \rceil-1=1$. Let $F=\{C_6\}$, $F$ is a $C_6$-structure cut and $2=\lceil \frac{4}{3} \rceil\leq\kappa(Q_4,C_6)\leq 1$ by Lemma \ref{TAMI}, a contradiction.
\end{proof}
By Lemma \ref{cycleevenks} and Lemma \ref{cyclesubstr}, we have the following theorem.
\begin{thm}\label{cycleksandkgeq}
	For $n\geq 3$ and even integer $3\leq k\leq 2^{n-1}$, $\kappa(Q_n;C_k)\geq \kappa^s(Q_n;C_k)=\lceil \frac{2n}{k}\rceil$.
\end{thm}

\begin{lem}\label{n=3}
	For $n=3$, it is obvious that $\kappa(Q_3;C_4)=2$.
\end{lem}

\begin{lem}\label{cubecycle}(See \cite{SaadTopological})
	A cycle of length $l$ can be mapped into $Q_n$ when $l$ is even and $4\leq l \leq 2^n$.
\end{lem}

\begin{lem}\label{oddpathcycle}
	For $n\geq 2$ and odd integer $1\leq q \leq 2^n-1$, each pair of adjacent vertices $u$ and $v$ in $Q_n$ have a path of length $q$ between them.
\end{lem}
\begin{proof}
	The lemma is true when $q=1$ since $u,v$ are adjacent. When $q\geq 3$, there is a cycle $C_{q+1}$ of length $q+1$ with $uv\in E(C_{q+1})$, $C_{q+1}-\{uv\}$ is the path  required.
\end{proof}

\begin{lem}\label{cycleleq}
	For $n\geq 4$ and $k$ even,
	\begin{equation}
	\kappa(Q_n;C_k)
	\begin{cases}
	=n-2 & \text{k=4,} \\
	\leq \lceil \frac{2n}{k} \rceil & \text{$n\geq 5$ and $6\leq k\leq 2^{n-2}$.}
	\end{cases}
	\end{equation}
\end{lem}
\begin{proof}
	For $k=4$, we have $\kappa(Q_n;C_4)=n-2$ by Theorem \ref{linsresult}.
	
    For $n\geq 5$ and $6\leq k\leq 2^{n-2}$,
	we set $u=00\dots0$ being a vertex in $Q_n$.
	\vskip .2cm
	{\bf Case 1.} $\frac{k}{2} \leq n$.
	
	\vskip .2cm
	{\bf Subcase 1.1.} $n$ is divided by $\frac{k}{2}$. Let $F = \{ C^i_k | i = 0,
	2, \ldots, \lceil \frac{2n}{k} \rceil - 1 = \frac{2n}{k} - 1 \}$ with $C^i_k
	= \left( (u)^{i \frac{k}{2}}, \left( (u)^{i \frac{k}{2}} \right)^{i
		\frac{k}{2} + 1}, (u)^{i \frac{k}{2} + 1}, \ldots, (u)^{(i + 1)
		\frac{k}{2} - 1}, \left( (u)^{(i + 1) \frac{k}{2} - 1} \right)^{i
		\frac{k}{2}}, (u)^{i \frac{k}{2}} \right)$. Then $F$ is a $C_k$-structure cut of $Q_n$.
	\begin{center}
		\begin{figure}[H] 
			\centering
			\def\svgwidth{250px}
%% Accompanies image file 'ill1.pdf' (pdf, eps, ps)
\begingroup%
  \makeatletter%
  \providecommand\color[2][]{%
    \errmessage{(Inkscape) Color is used for the text in Inkscape, but the package 'color.sty' is not loaded}%
    \renewcommand\color[2][]{}%
  }%
  \providecommand\transparent[1]{%
    \errmessage{(Inkscape) Transparency is used (non-zero) for the text in Inkscape, but the package 'transparent.sty' is not loaded}%
    \renewcommand\transparent[1]{}%
  }%
  \providecommand\rotatebox[2]{#2}%
  \newcommand*\fsize{\dimexpr\f@size pt\relax}%
  \newcommand*\lineheight[1]{\fontsize{\fsize}{#1\fsize}\selectfont}%
  \ifx\svgwidth\undefined%
    \setlength{\unitlength}{322.91689028bp}%
    \ifx\svgscale\undefined%
      \relax%
    \else%
      \setlength{\unitlength}{\unitlength * \real{\svgscale}}%
    \fi%
  \else%
    \setlength{\unitlength}{\svgwidth}%
  \fi%
  \global\let\svgwidth\undefined%
  \global\let\svgscale\undefined%
  \makeatother%
  \begin{picture}(1,0.59352892)%
    \lineheight{1}%
    \setlength\tabcolsep{0pt}%
    \put(0,0){\includegraphics[width=\unitlength,page=1]{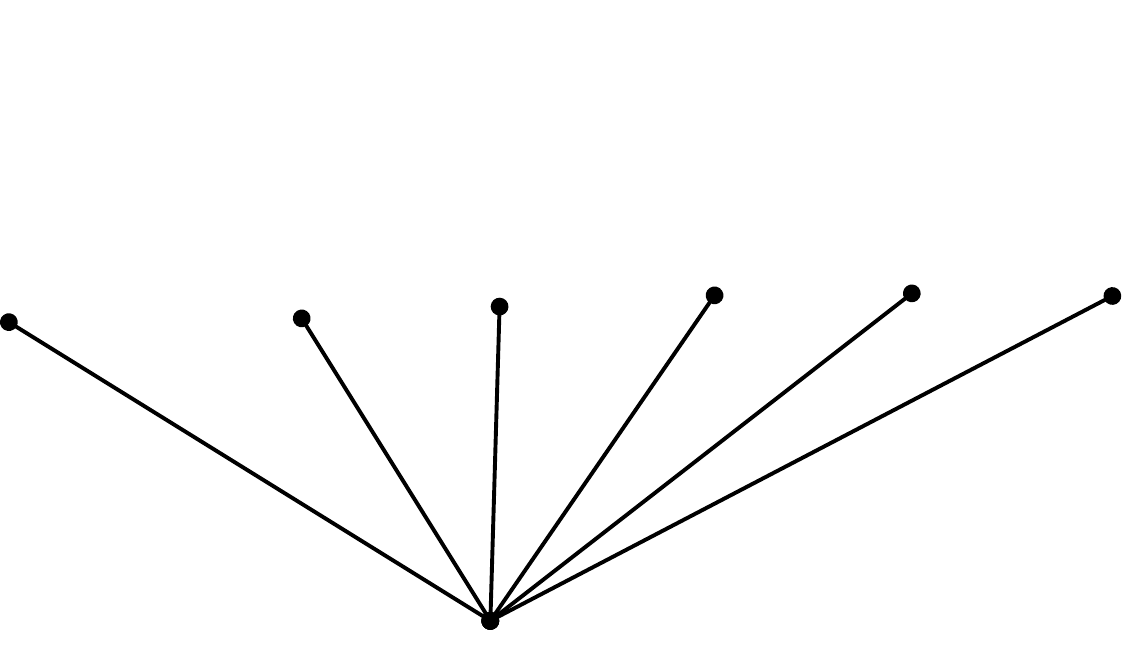}}%
    \put(0.42518154,0.0049128){\color[rgb]{0,0,0}\makebox(0,0)[lt]{\lineheight{1.25}\smash{\begin{tabular}[t]{l}$u$\end{tabular}}}}%
    \put(0,0){\includegraphics[width=\unitlength,page=2]{ill1.pdf}}%
  \end{picture}%
\endgroup%

		\end{figure}
		\begin{center}
			\textbf{Fig 7.} An example for $n=6$ and $k=6$.
		\end{center}
	\end{center}
	
	\vskip .2cm
	{\bf Subcase 1.2.} $n$ is not divided by $\frac{k}{2}$. Let $F = \{ C^i_k | i =
	0, 2, \ldots, \lceil \frac{2n}{k} \rceil - 1 \}$ with $C^i_k = \left(
	(u)^{i \frac{k}{2}}, \left( (u)^{i \frac{k}{2}} \right)^{i \frac{k}{2}
		+ 1}, (u)^{i \frac{k}{2} + 1}, \ldots, (u)^{(i + 1) \frac{k}{2} - 1},
	\left( (u)^{(i + 1) \frac{k}{2} - 1} \right)^{i \frac{k}{2}}, (u)^{i
		\frac{k}{2}} \right)$ for $i = 0, \ldots, \lceil \frac{2n}{k} \rceil - 2$ and
	$$C^{\lceil \frac{2n}{k} \rceil - 1}_k = \left( (u)^{n - \frac{k}{2}},
	\left( (u)^{n - \frac{k}{2}} \right)^{n - \frac{k}{2} + 1}, (u)^{n -
		\frac{k}{2} + 1}, \ldots, (u)^{n - 1}, \left(\left(u\right)^{n - 1}\right)^{n -
		\frac{k}{2}}, (u)^{n - \frac{k}{2}} \right) .$$ Then $F$ is a $C_k$-structure cut of $Q_n$.
	
\begin{center}
	\begin{figure}[H] 
		\centering
		\def\svgwidth{250px}
%% Accompanies image file 'ill2.pdf' (pdf, eps, ps)
\begingroup%
  \makeatletter%
  \providecommand\color[2][]{%
    \errmessage{(Inkscape) Color is used for the text in Inkscape, but the package 'color.sty' is not loaded}%
    \renewcommand\color[2][]{}%
  }%
  \providecommand\transparent[1]{%
    \errmessage{(Inkscape) Transparency is used (non-zero) for the text in Inkscape, but the package 'transparent.sty' is not loaded}%
    \renewcommand\transparent[1]{}%
  }%
  \providecommand\rotatebox[2]{#2}%
  \newcommand*\fsize{\dimexpr\f@size pt\relax}%
  \newcommand*\lineheight[1]{\fontsize{\fsize}{#1\fsize}\selectfont}%
  \ifx\svgwidth\undefined%
    \setlength{\unitlength}{326.12661306bp}%
    \ifx\svgscale\undefined%
      \relax%
    \else%
      \setlength{\unitlength}{\unitlength * \real{\svgscale}}%
    \fi%
  \else%
    \setlength{\unitlength}{\svgwidth}%
  \fi%
  \global\let\svgwidth\undefined%
  \global\let\svgscale\undefined%
  \makeatother%
  \begin{picture}(1,0.58816589)%
    \lineheight{1}%
    \setlength\tabcolsep{0pt}%
    \put(0,0){\includegraphics[width=\unitlength,page=1]{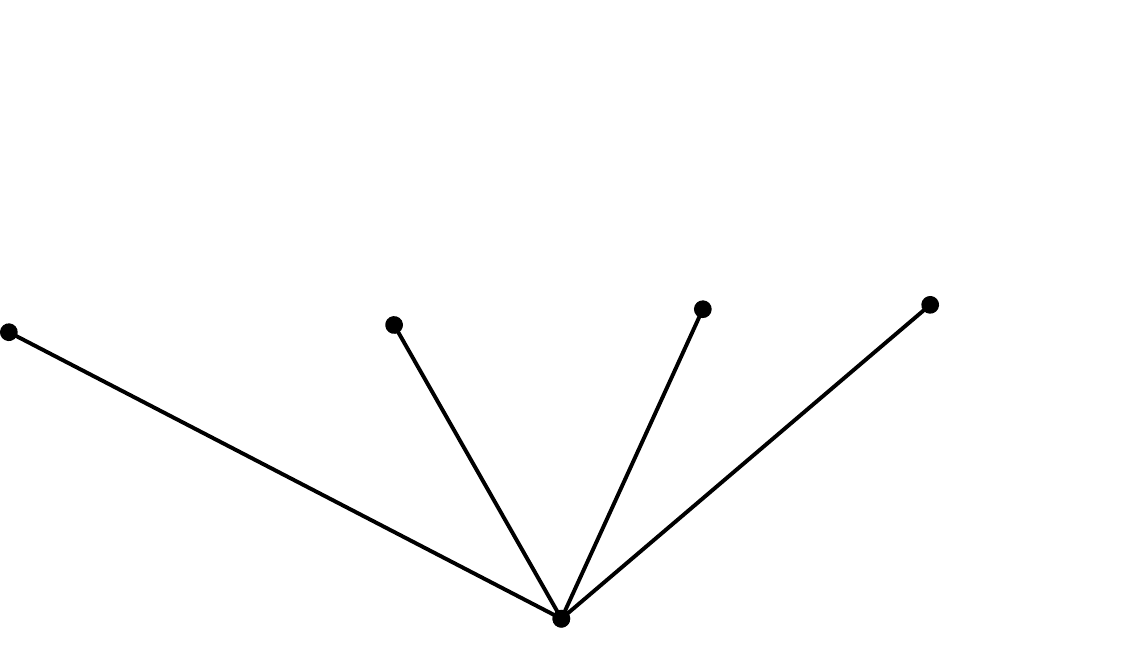}}%
    \put(0.49283671,0.00486445){\color[rgb]{0,0,0}\makebox(0,0)[lt]{\lineheight{1.25}\smash{\begin{tabular}[t]{l}$u$\end{tabular}}}}%
    \put(0,0){\includegraphics[width=\unitlength,page=2]{ill2.pdf}}%
  \end{picture}%
\endgroup%

	\end{figure}
	\begin{center}
		\textbf{Fig 8.} An example for $n=5$ and $k=6$.
	\end{center}
\end{center}
	
	\vskip .2cm
	{\bf Case 2.} $\frac{k}{2} \geq n+1$. In this case $\lceil\frac{2n}{k}\rceil=1$ and $k\geq 2n+2$. Let $Q$ be the induced subgraph of $Q_n$ with $V(Q)=\{v=x_0x_1\dots x_{n-3}01|x_i=0\text{ or }1,i=0,1,\dots,n-3 \}$, obviously $\left((u)^0\right)^{n-1}$ and $(u)^{n-1}$ are belong to $V(Q)$, by Lemma \ref{oddpathcycle}, there is a path $P$ of length $k-(2n-1)$ between $\left((u)^0\right)^{n-1}$ and $(u)^{n-1}$ in $Q$. Let $F=\{C_k\}$ with $$C_k = \left( (u)^0, \left(\left(u\right)^0\right)^1,
	\ldots, \left(\left(u\right)^{n - 2}\right)^{n - 1}, (u)^{n - 1}, P, \left(\left(u\right)^0\right)^{n - 1}, (u)^0\right)$$Then $F$ is a $C_k$-structure cut of $Q_n$.
	\begin{center}
	\begin{figure}[H] 
		\centering
		\def\svgwidth{300px}
		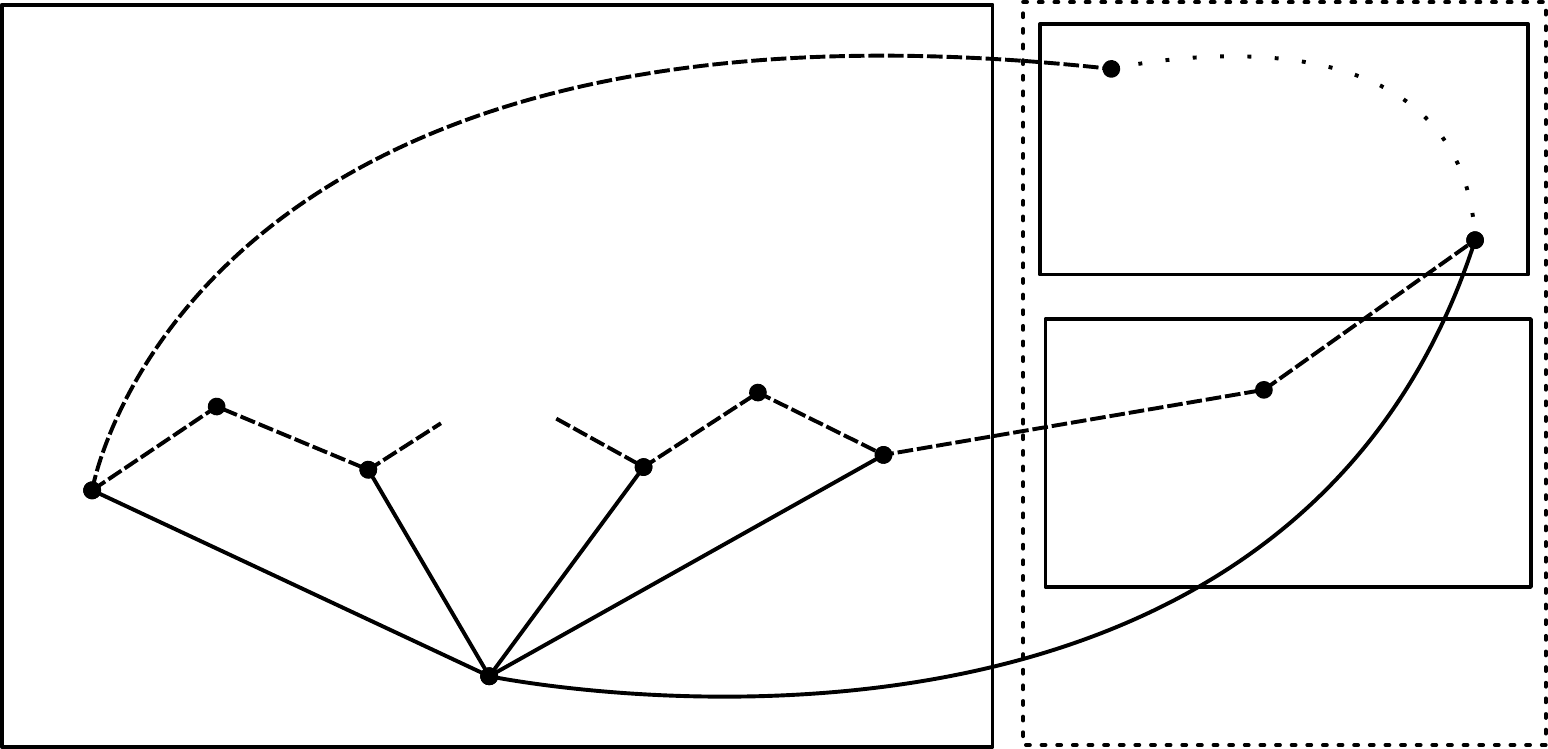
	\end{figure}
	\begin{center}
		\textbf{Fig 9.}
	\end{center}
\end{center}
\end{proof}

By Lemma \ref{cycleoddks}, Lemma \ref{cycleksandkgeq}, Lemma \ref{n=3} and Lemma \ref{cycleleq} we can obtain the following theorem.
\begin{thm}\label{cycleres}
	For $n\geq 3$
	\begin{equation}
	\kappa^s(Q_n;C_k)=
	\begin{cases}
	\lceil\frac{2n}{k+1}\rceil & \text{$k$ is odd and $3\leq k\leq 2^{n-1}$,} \\
	\lceil\frac{2n}{k}\rceil & \text{$k$ is even and $4\leq k\leq 2^{n-1}$.}
	\end{cases}
	\end{equation}
	\begin{equation}
	\kappa(Q_n;C_k)
	\begin{cases}
	=2 & \text{$n=3,k=4$,} \\
	=n-2 & \text{$n\geq 4$ and $k=4$,}\\
	=\lceil\frac{2n}{k}\rceil & \text{$n\geq 5$, $k$ is even and $6\leq k\leq 2^{n-2}$,}\\
	\geq \lceil\frac{2n}{k}\rceil & \text{$n\geq 4$, $k$ is even and $2^{n-2}+2\leq k\leq 2^{n-1}$.}
	\end{cases}
	\end{equation}
\end{thm}

\begin{lem}\label{kks2^m}
	For $n=4,5$ and $2\leq m\leq n-2$, $\kappa(Q_n;C_{2^m})=n-m$.
\end{lem}
\begin{proof}
	We have $\kappa(Q_n;C_{2^m})\geq \lceil\frac{n}{2^{m-1}}\rceil$ by Lemma \ref{cycleksandkgeq}. We also have $\kappa(Q_n;C_{2^m})\leq n-m$ by Lemma \ref{Manesresult}.
	\begin{center}
		\begin{table}[H]
			\centering
			\begin{tabular}{|c|c|c|}
				\hline
				& $\lceil\frac{n}{2^{m-1}}\rceil$ & $n-m$ \\ \hline
				$n=4,m=2$ & 2                               & 2     \\ \hline
				$n=5,m=2$ & 3                               & 3     \\ \hline
				$n=5,m=3$ & 2                               & 2     \\ \hline
			\end{tabular}
		\end{table}
	\end{center}

Therefore the statement is proved.
\end{proof}

\begin{lem}\label{kks2mcy}
	For $n\geq 6$ and $3\leq m\leq n-2$, $\kappa(Q_n;C_{2^m})=\kappa^s(Q_n;C_{2^m})=\lceil\frac{n}{2^{m-1}}\rceil$
\end{lem}
\begin{proof}
	By Theorem \ref{cycleres}.
\end{proof}
\begin{lem}\label{budengs}
	For $n\geq 6$ and $3\leq m\leq n-2$, $\lceil\frac{n}{2^{m-1}}\rceil<n-m$.
\end{lem}
\begin{proof}
	For a given integer $n\geq 6$, let $f (m) =\frac{n}{2^{m-1}} + 1 - (n- m)$. Derivate $f(m)$ with respect to $m$, $f' (m) = 1 - 2^{1 - m} n \ln 2$, again derivate $f'(m)$, we have $f'' (m) = 2^{1 -m} n (\ln 2)^2 $. Obviously, $f'' (m) \geq 0$ which implies $f'(m)$ monotonically increasing respect to $m$.
	
   After simple calcutions we obtain 
	\begin{equation}
	\begin{cases}
	f' (3) = 1 - \frac{n \ln 2}{4} \\
	f' (n - 2) =
	1 - \frac{n \ln 2}{2^{n - 3}}
	\end{cases}
	\end{equation}
	$f' (3)<0$ when $n\geq 6$. Let $g (n) =f' (n - 2)= 1
	- \frac{n \ln 2}{2^{n - 3}}$, derivate $g(n)$ with respect to $n$, $g' (n) = - 2^{3 - n} \ln 2 (1
	- nln2) >0$ when $n\geq 6$, therefore $g(n)$ monotonically increasing respect to $n$ and $\min\limits_{n\geq 6} g(n)=g(6)=1 - \frac{6 \ln
		2}{8} > 0$, this implies $f' (n - 2)=g (n)>0$ and $f(m)$ firstly monotonically decreasing and then monotonically increasing respect to $m$. Therefore $\max\limits_{m} f (m) =\max\limits_{m}\{ f (3), f (n - 2) \}$.
	
	By simple calculations we obtain
	\begin{equation}
	\begin{cases}
	f (3) = 4 - \frac{3 n}{4} < 0 \\
	f (n - 2) = \frac{n}{2^{n - 3}} - 1
	\end{cases}
	\end{equation}
	Let $h (n) = \frac{n}{2^{n - 3}} - 1$, derivate $h(n)$ with respect to $n$, $h' (n) = 2^{3 - n} (1- nln2) <0$ when $n\geq 6$. Therefore $h(n)$ monotonically decreasing respect to $n$ and $\max\limits_{n\geq 6}h(n)=h(6)=\frac{6}{8} - 1 <
	0$, therefore $f (n - 2) < 0$ implies $f (m) = \frac{n}{2^{m-1}} + 1 - (n - m) < 0$ i.e. $\lceil \frac{n}{2^{m-1}} \rceil \leq \frac{n}{2^{m-1}} + 1 < n - m$.
\end{proof}

By Theorem \ref{cycleres}, Lemma \ref{kks2^m}, Lemma \ref{kks2mcy} and Lemma \ref{budengs}, we obtain the following theorem as the solution of the open problem in \cite{ManeStructure}.

\begin{thm}
	For $n\geq 4$ and for each $2\leq m \leq n-2$
	\begin{equation}
	\kappa(Q_n;C_{2^m})=
	\begin{cases}
	=n-m& \text{$n=4,5$ or $n\geq 4,m=2$,} \\
	=\lceil\frac{n}{2^{m-1}}\rceil<n-m& \text{$n\geq 6$ and $3\leq m\leq n-2$.}
	\end{cases}
	\end{equation}
\end{thm}

In Theorem \ref{cycleres} we have $\kappa(Q_n;C_k)\geq \lceil\frac{2n}{k}\rceil$ for $n\geq 4$, $k$ even and $2^{n-2}+2\leq k\leq 2^{n-1}$, but in this case, whether $\kappa(Q_n;C_k)$ equals $\lceil\frac{2n}{k}\rceil$ is still a question. 
\section*{Acknowledgements}
The project is supported partially by Hu Xiang Gao Ceng Ci Ren Cai Ju Jiao Gong Cheng-Chuang Xin Ren Cai (No. 2019RS1057).

\bibliographystyle{plain}

\bibliography{refer}

\begin{thebibliography}{1}

\bibitem{TAMIZHCHELVAM2019}
T.~Tamizh Chelvam and M.~Sivagami.
\newblock Structure and substructure connectivity of circulant graphs and
  hypercubes.
\newblock {\em Arab Journal of Mathematical Sciences}, 2019.

\bibitem{diestel2012graph}
Reinhard Diestel.
\newblock {\em Graph Theory: Springer Graduate Text GTM 173}, volume 173.
\newblock Reinhard Diestel, 2012.

\bibitem{LI2019169}
Dong Li, Xiaolan Hu, and Huiqing Liu.
\newblock Structure connectivity and substructure connectivity of twisted
  hypercubes.
\newblock {\em Theoretical Computer Science}, 796:169 -- 179, 2019.

\bibitem{LinStructure}
Cheng~Kuan Lin, Lili Zhang, Jianxi Fan, and Dajin Wang.
\newblock Structure connectivity and substructure connectivity of hypercubes.
\newblock {\em Theoretical Computer Science}, 634:97--107.

\bibitem{ManeStructure}
S.~A. Mane.
\newblock Structure connectivity of hypercubes.
\newblock {\em Akce International Journal of Graphs \& Combinatorics},
  15(1):49--52.

\bibitem{SaadTopological}
Y.~Saad and M.H. Schultz.
\newblock Topological properties of hypercubes.
\newblock {\em Computers IEEE Transactions on}, 37(7):867--872.

\bibitem{YangExtraconnectivity}
Weihua Yang and Jixiang Meng.
\newblock Extraconnectivity of hypercubes.
\newblock {\em Applied Mathematics Letters}, 22(6):887--891.

\bibitem{ZHOU2017208}
Jin-Xin Zhou.
\newblock On g-extra connectivity of hypercube-like networks.
\newblock {\em Journal of Computer and System Sciences}, 88:208 -- 219, 2017.

\end{thebibliography}

\end{document}